\documentclass[smallextended]{svjour3}       
\smartqed  
\usepackage{graphicx}
\usepackage{amsmath}
\usepackage{amsfonts}
\usepackage{subfigure}
\usepackage{tikz}
\usetikzlibrary{decorations.markings}
\usetikzlibrary{plotmarks}
\usepackage{IEEEtrantools}
\makeatletter
\newcommand*{\rom}[1]{\expandafter\@slowromancap\romannumeral #1@}

\newcommand{\ie}{{\it i.e.},\ }

\usepackage[T1]{fontenc}
\usepackage[utf8]{inputenc}
\usepackage{authblk}
\usepackage{tikz}
\usepackage{subfigure}
\usepackage{pgfplots}
\usetikzlibrary{decorations.markings}
\usetikzlibrary{plotmarks}
\usepackage{graphics}
\usepackage{float}
\usepackage{amssymb}
\usepackage{amsfonts}
\usepackage{amsmath}
\usepackage{array}
\usepackage{IEEEtrantools}
\usepackage{graphicx}
\usepackage{tabularx}

\makeatletter
\newcommand{\thickhline}{%
    \noalign {\ifnum 0=`}\fi \hrule height 1pt
    \futurelet \reserved@a \@xhline
}
\newcolumntype{"}{@{\hskip\tabcolsep\vrule width 1pt\hskip\tabcolsep}}
\makeatother

\pgfplotsset{tick label style={font=\Huge},    label style={font=\Huge},    legend style={font=\Huge}}


\usepackage[T1]{fontenc}
\usepackage[utf8]{inputenc}
\usepackage{authblk}
\usepackage{tikz}
\usepackage{subfigure}
\usepackage{pgfplots}
\usetikzlibrary{decorations.markings}
\usetikzlibrary{plotmarks}
\usepackage{graphics}
\usepackage{float}
\usepackage{amssymb}
\usepackage{amsfonts}
\usepackage{amsmath}
\usepackage{array}
\usepackage{IEEEtrantools}

\usepackage{tikz} 

\usetikzlibrary{positioning}
\usepackage{pgfplots}

\usepackage{array}
\usepackage{mathtools}

\usepackage{IEEEtrantools}
\IEEEeqnarraydefcol{eqnarrayCscript}{\hfil$\scriptstyle}{$\hfil}

\begin{document}

\title{Invariant measures and error bounds for random walks in the quarter-plane based on sums of geometric terms
}

\titlerunning{Invariant measures and error bounds for random walks in the quarter-plane}        

\author{Yanting Chen         \and
        Richard J. Boucherie \and
        Jasper Goseling 
}


\institute{Y. Chen \and R.J. Boucherie \and J. Goseling \at
              Stochastic Operations Research, University of Twente, 
P.O. Box 217, 
7500 AE Enschede, 
The Netherlands\\
              \email{\{y.chen, r.j.boucherie, j.goseling\}@utwente.nl}           
}

\date{Received: date / Accepted: date}

\maketitle

\begin{abstract}
We consider homogeneous random walks in the quarter-plane. The necessary conditions which characterize random walks of which the invariant measure is a sum of geometric terms are provided in~\cite{chen2013necessary,chen2012invariant}. Based on these results, we first develop an algorithm to check whether the invariant measure of a given random walk is a sum of geometric terms. We also provide the explicit form of the invariant measure if it is a sum of geometric terms. Secondly, for random walks of which the invariant measure is not a sum of geometric terms, we provide an approximation scheme to obtain error bounds for the performance measures. Finally, some numerical examples are provided.
\end{abstract}

\section{Introduction}
\label{sec:introD}
Random walks in the quarter-plane serve as the underlying models for many two-node queueing systems. It is of great interest to find performance measures, either exactly or approximately, of such systems.


If the invariant measure of the random walk is known in closed-form, then the performance measures can be computed directly. The canonical example is the Jackson network of which the invariant measure is of product-form~\cite[Chapter 6]{wolff1989stochastic}. For some random walks the invariant measure can also be expressed as a linear
combination of countably many geometric terms~\cite{adan1993compensation}. However, if the invariant measure is not of closed-form, then closed-form performance measures are often not available. Various approaches to finding the invariant measure of a random walk in the quarter-plane exist. Most notably, methods from complex analysis have been used to obtain the generating function of the invariant measure~\cite{cohen1983boundary,fayolle1999random}. Matrix-geometric methods provide an algorithmic approach to finding the invariant measure~\cite{neuts1981matrix}. However, explicit closed-form expressions for the invariant measures of random walks are difficult to obtain using the methods mentioned above. Hence, it is, in general, not possible to find exact results for the performance measures of random walks in the quarter-plane. 

Our first contribution in this paper is to characterize the class of random walks for which the invariant measures can be expressed in closed-form. In particular, based on results from~\cite{chen2013necessary,chen2012invariant}, we characterize the random walks for which the invariant measure is a sum of finitely many geometric terms. Similar to the evaluation of the random walks of which the invariant measures are of product-form, the performance measures of such systems can be readily evaluated.
For any given random walk, we provide an algorithm to detect whether its invariant measure is a sum of geometric terms. Moreover, we also explain how to obtain this closed-form invariant measure explicitly, if it exists.

When a closed-form invariant measure for a random walk in the quarter-plane does not exist, we usually have to find approximations for the performance measures we are interested in. Van Dijk et al.~\cite{vandijk11inbook,vandijk88tandem,vandijk88perturb} developed a perturbation theory to approximate the performance of a queueing system by relating it to the performance of a perturbed queueing system of which the stationary distribution is of product-form. The method of van Dijk et al. relies on carefully constructing the modifications that provide the perturbed random walk. Goseling et al.~\cite{goseling2014linear} expressed the upper or lower bound of a performance measure as the value of the objective function of the optimal solution of a linear program. This method generalizes the model modification approach based on the perturbed random walk developed in~\cite{vandijk88tandem,vandijk88perturb,vandijk11inbook} and it accepts any random walk in the quarter-plane as an input. However, the random walk that is used as the perturbed random walk in the approximation scheme of~\cite{goseling2014linear} is still restricted to have an invariant measure that is of product-form, as we will see from some examples that in this paper, large perturbation might be required to obtain a product-form invariant measure. This prevents us from having good approximations for some random walks in the quarter-plane. Our second contribution is to establish an approximation scheme similar to that in~\cite{goseling2014linear}, where the perturbed random walk is allowed to have a sum of geometric terms invariant measure. With this extension, the approximations for performance measures will be improved because we have a larger candidate set for the perturbed random walks. Numerical results also illustrate that better approximations are achieved if we consider a richer candidate set for the perturbed random walks.

The remainder of this paper proceeds as follows. In Section~\ref{sec:modelD}, we present the model and definitions. In Section~\ref{sec:algorithmD}, we provide an algorithm that checks whether the invariant measure of a given random walk is a sum of geometric terms. In Section~\ref{sec:LPapproximationD}, we provide an approximation scheme to bound the performance measures when the invariant measure of the given random walk cannot be a sum of geometric terms.  We consider several examples and show numerical results in Section~\ref{sec:examplesD}. In Section~\ref{sec:discussionD}, we summarize our results and shortly discuss extensions of our approximation scheme.

\section{Model and problem statement} \label{sec:modelD}

%
We consider a two-dimensional random walk $R$ on the pairs  of non-negative integers, \ie $S = \{(i,j), i,j \in \mathbb{N}_0\}$. We refer to $\{(i,j) | i>0, j>0\}$, $\{(i,j) | i>0, j=0\}$, $\{(i,j) | i=0, j>0\}$ and $(0,0)$ as the interior, the horizontal axis, the vertical axis and the origin of the state space, respectively. The transition probability from state $(i,j)$ to state $(i+s,j+t)$ is denoted by $p_{s,t}(i,j)$. Transitions are restricted to the adjoined points (horizontally, vertically and diagonally), \ie $p_{s,t}(k,l)=0$ if $|s|>1$ or $|t|>1$. The random walk is homogeneous in the sense that for each pair $(i,j)$, $(k,l)$ in the interior (respectively on the horizontal axis and on the vertical axis) of the state space
\begin{equation} \label{eq:homogeneousD}
p_{s,t}(i,j)=p_{s,t}(k,l)\quad\text{and}\quad p_{s,t}(i-s,j-t)=p_{s,t}(k-s,l-t),
\end{equation}
for all $-1\leq s\leq 1$ and $-1\leq t\leq 1$.  We introduce, for $i>0$, $j>0$, the notation $p_{s,t}(i,j)=p_{s,t}$, $p_{s,0}(i,0)=h_s$ and $p_{0,t}(0,j)=v_t$.
Note that the first equality of~\eqref{eq:homogeneousD} implies that the transition probabilities for each part of the state space are translation invariant. The second equality ensures that also the transition probabilities entering the same part of the state space are translation invariant. The above definitions imply that $p_{1,0}(0,0)=h_1$ and $p_{0,1}(0,0)=v_1$. The model and notations are illustrated in Figure~\ref{fig:rwD}. 
\begin{figure}
\includegraphics{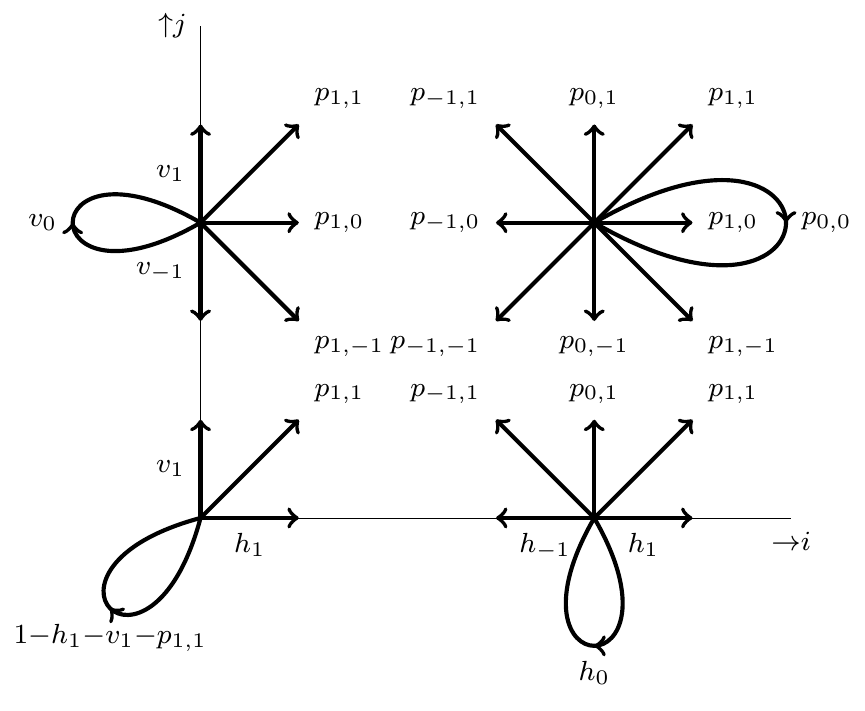}
\caption{Random walk in the quarter-plane. \label{fig:rwD}}
\end{figure}

All the random walks which we consider in this paper are assumed to be irreducible, aperiodic and positive recurrent. Moreover, we assume that at least one of the transition probabilities to North, Northeast or East is non-zero, \ie $p_{0,1} + p_{1,1} + p_{0,1} \neq 0$. The reason is that the case $p_{0,1} + p_{1,1} + p_{0,1} = 0$ which is necessary for applying the compensation approach has been discussed in detail in~\cite{adan1993compensation}.

Let $m$: $S \rightarrow [0,\infty)$ denote the invariant probability measure of $R$, \ie for $i > 0$ and $j > 0$,
\begin{align}
m(i,j) &= \sum_{s = -1}^{1} \sum_{t = -1}^{1} m(i-s,j-t)p_{s,t}, \label{eq:interiorD} \\
m(i,0) &= \sum_{s = -1}^{1} m(i-s,1) p_{s,-1} + \sum_{s = -1}^{1} m(i-s,0)h_s, \label{eq:horizontalD} \\
m(0,j) &= \sum_{t = -1}^{1} m(1,j-t) p_{-1,t} + \sum_{t = -1}^{1} m(0,j-t)v_t. \label{eq:verticalD} 
\end{align}
We will refer to the above equations as the balance equations in the interior, the horizontal axis and the vertical axis of the state space. The balance equation at the origin is implied by the balance equations of all other states and is, therefore, not considered.

Our interest is in the steady-state performance of $R$. The performance measures that we consider are induced by functions that are linear in each part of the state space, \ie linear in the interior, on the horizontal axis and on the vertical axis. More formally, we consider the performance measure $\mathcal{F}$, defined as
\begin{equation} \label{eq:FdefD}
\mathcal{F} = \sum_{(i,j) \in S} m(i,j) F(i,j),
\end{equation}
where $F: S \rightarrow [0,\infty)$ is defined as
\begin{equation}{\label{eq:performancemeasureD}}
F(i,j) =
\begin{cases}
f_{1,0} + f_{1,1} i, \quad &\text{if } i>0\text{ and }j=0,\\
f_{2,0} + f_{2,2} j, \quad &\text{if } i=0\text{ and }j>0,\\
f_{3,0}, \quad &\text{if } i=j=0,\\
f_{4,0} + f_{4,1} i + f_{4,2} j, \quad &\text{if } i>0\text{ and }j>0,
\end{cases}
\end{equation}
and the $f_{p,q}$ are constants that define the function.

If $m$ is a sum of geometric terms, then the performance measure $\mathcal{F}$ can be immediately obtained from~\eqref{eq:FdefD}. In Section~\ref{sec:sumsD} we will introduce such an invariant measure, \ie a sum of geometric terms. In addition, we provide a complete characterization of the random walk of which the invariant measure is a sum of geometric terms. We also provide an algorithm to detect whether the invariant measure of a given random walk is a sum of geometric terms.

For the random walk of which the invariant measure is not a sum of geometric terms, we resort to deriving lower and upper bounds on $\mathcal{F}$. These bounds are constructed in Section~\ref{sec:LPapproximationD}. Measures that are a sum of geometric terms, defined next in Section~\ref{sec:sumsD}, form the basis for these bounds.

\section{Random walks with an invariant measure that is a sum of geometric terms}{\label{sec:algorithmD}} \label{sec:sumsD}

In this section, we will see that not all linear combination of geometric terms may yield an invariant measure of a random walk. We first characterize the linear combination of geometric terms that can be the invariant measure of a random walk. Then, we provide an algorithm to check whether the invariant measure of the given random walk is a sum of geometric terms. We apply this algorithm to several random walks. Finally, we show that the running time for the algorithm is finite. We also explain how to obtain the invariant measure explicitly, if it is a sum of geometric terms.

\subsection{Preliminaries} \label{sec:defD}
We are interested in measures that can be expressed as a linear combination of geometric measures. We first introduce the following geometric measure.
\begin{definition}[Geometric measure] \label{def:geomD}
The measure $m(i,j)$ is a geometric measure if $m(i,j) = \rho^i \sigma^j$ for some $(\rho, \sigma) \in (0,1)^2$. 
\end{definition}
We represent a geometric measure $\rho^i \sigma^j$ by its coordinate $(\rho, \sigma)$ in $[0, 1)^2$. Then $\Gamma \subset [0,\infty)^2$ characterizes a set of geometric measures. To identify the geometric measures that satisfy the balance equations in the interior, on the horizontal axis and on the vertical axis of the state space, we introduce the polynomials
\begin{align}
Q(x,y) &= xy\left(\sum_{s = -1}^{1} \sum_{t = -1}^{1} x^{-s} y^{-t} p_{s,t} - 1\right),   {\label{eq:QfunctionD}}  \\
H(x,y) &= xy\left(\sum_{s = -1}^{1}  x^{-s} h_{s} + y\left(\sum_{s = -1}^{1} x^{-s} p_{s,-1}\right) - 1\right),  {\label{eq:HfunctionD}}  \\
V(x,y) &= xy\left(\sum_{t = -1}^{1}  y^{-t} v_{t} + x\left(\sum_{t = -1}^{1} y^{-t} p_{-1,t}\right) - 1\right),  {\label{eq:VfunctionD}}  
\end{align}
respectively. For example, $Q(\rho,\sigma) = 0$, $H(\rho, \sigma) = 0$ and $V(\rho, \sigma) = 0$ imply that $m(i,j) = \rho^i \sigma^j$, $(i,j) \in S$ satisfies~\eqref{eq:interiorD},~\eqref{eq:horizontalD} and~\eqref{eq:verticalD}, respectively. If all balance equations are satisfied, then $m(i,j)$ is the invariant measure of the random walk $R$. Let the curves $Q$, $H$ and $V$ denote the sets of $(x,y)$ restricted to $[0, \infty)^2$, satisfying $Q(x,y) = 0$, $H(x,y) = 0$ and $V(x,y) = 0$.

Next, we analyze the measures that are sums of geometric measures.
\begin{definition}[{Induced measure}]\label{def:inducedD}
The measure $m$ is called induced by $\Gamma\subset (0, \infty)^2$ if
\begin{equation*}
m(i,j) = \sum_{(\rho, \sigma) \in \Gamma} \alpha(\rho, \sigma)\rho^i\sigma^j,
\end{equation*}
with $\alpha(\rho, \sigma) \in \mathbb{R}\backslash \{0\}$ for all $(\rho, \sigma) \in \Gamma$.
\end{definition}

We have excluded in Definitions~\ref{def:geomD} and~\ref{def:inducedD} the case where $\rho=0$ or $\sigma=0$, \ie the case of degenerate geometric measures. The reason is that it has been shown in~\cite{chen2012invariant} that linear combinations that contain degenerate geometric measures cannot be the invariant measure for any random walk.

Since we have assumed that $p_{1,0} + p_{1,1} + p_{0,1} \neq 0$, we know from~\cite{chen2013necessary,chen2012invariant} that in this case, $|\Gamma| < \infty$, \ie there are finitely many geometric terms in the set $\Gamma$, if 
\begin{equation*}
m(i,j) = \sum_{(\rho,\sigma) \in \Gamma} \alpha(\rho, \sigma)\rho^i\sigma^j
\end{equation*}
is the invariant measure of the random walk.

\subsection{Characterization}
We first provide conditions which the set $\Gamma$ must satisfy such that the induced measure of $\Gamma$ may be the invariant measure of a random walk. This result provides the theoretical support for the Detection Algorithm that will be presented in the next subsection. The Detection Algorithm determines whether the invariant measure of a given random walk is a sum of geometric terms.

The results in this and subsequent sections are based on the notion of uncoupled partitions of $\Gamma$ and pairwise-coupled set, which were first introduced in~\cite{chen2012invariant}.

\begin{definition}[Uncoupled partition]
A partition $\{\Gamma_1, \Gamma_2, \cdots\}$ of $\Gamma$ is \emph{horizontally uncoupled} if $(\rho, \sigma) \in \Gamma_p$ and $(\tilde{\rho}, \tilde{\sigma}) \in \Gamma_q$ for $p \neq q$, implies that $\tilde{\rho} \neq \rho$, \emph{vertically uncoupled} if $(\rho, \sigma) \in \Gamma_p$ and $(\tilde{\rho}, \tilde{\sigma}) \in \Gamma_q$ for $p \neq q$, implies $\tilde{\sigma} \neq \sigma$, and \emph{uncoupled} if it is both horizontally and vertically uncoupled.
\end{definition}
We call a partition with the largest number of sets a maximal partition. It has been shown in~\cite{chen2012invariant} that the maximal horizontally uncoupled partition, the maximal vertically uncoupled partition and the maximal uncoupled partition are unique.

\begin{definition}[Pairwise-coupled set]
A set $\Gamma$ is pairwise-coupled if and only if the maximal uncoupled partition of $\Gamma$ contains only one set.
\end{definition}

Let $H_{set}$ and $V_{set}$ be the intersections of $Q$ with $H$ and $V$, which are restricted to the unit square, \ie
\begin{align}
 H_{set} &= \{(x,y) \in (0,1)^2 | (x,y) \in Q \cap H\}, \\
 V_{set} &= \{(x,y) \in (0,1)^2 | (x,y) \in Q \cap V\}.
\end{align}
We first present the following lemma.

\begin{lemma}{\label{lem:HVsetD}}
We have $|H_{set}| \leq 3$ and $|V_{set}| \leq 3$.
\end{lemma}
\begin{proof}
Without loss of generality, we only consider the intersections of $Q$ and $H$. It can be readily verified that the horizontal coordinates of the intersections of $Q$ with $H$ are the solutions of a polynomial of degree $4$ equating $0$ by combining Equations~\eqref{eq:interiorD} and~\eqref{eq:horizontalD}. Moreover, it is easy to verify that $(1, \frac{\sum_{s = -1}^{1} p_{s,1}}{\sum_{s = -1}^{1} p_{s,-1}})$ is an intersection of $Q$ and $H$. 
\end{proof}

We are now ready to present the next theorem, the proof of which is given in Appendix~\ref{app:oddD}. The theorem and its proof are built on the results from~\cite{chen2013necessary,chen2012invariant}. The theorem characterizes the sets $\Gamma$ for which the measure induced by $\Gamma$ may be the invariant measure of a random walk. The result states that $\Gamma$ must be a pairwise-coupled set connecting two geometric terms from set $H_{set} \cup V_{set}$. If both geometric terms are from $H_{set}$ or both of them are from $V_{set}$, then the pairwise-coupled set contains an even number of geometric terms. If one of the geometric terms is from $H_{set}$ and the other is from $V_{set}$, then the pairwise-coupled set contains an odd number of geometric terms.

We exclude the case where $|\Gamma| = 1$ here, \ie we do not consider the product-form invariant measures, because the characterization of this case has already been extensively studied.
\begin{theorem}{\label{thm:oddD}}
If the invariant measure of the random walk $R$ is induced by a set $\Gamma$, where $1 < |\Gamma| < \infty$, then $\Gamma$ is pairwise-coupled and there exist unique $(\rho_1, \sigma_1), (\rho_2, \sigma_2) \in \Gamma$ where $(\rho_1, \sigma_1) \neq (\rho_2, \sigma_2)$ such that
\begin{enumerate}
\item  $(\rho_1, \sigma_1), (\rho_2, \sigma_2) \in H_{set} \cup V_{set}$.
\item For $k = 1,2$, if $(\rho_k, \sigma_k) \in H_{set}$, then there exist a $(\rho, \sigma) \in \Gamma \backslash (\rho_k, \sigma_k)$ such that $\sigma = \sigma_k$ and there does not exist a $(\rho, \sigma) \in \Gamma \backslash (\rho_k, \sigma_k)$ such that $\rho = \rho_k$. Similarly, if $(\rho_k, \sigma_k) \in V_{set}$, then there exists a $(\rho, \sigma) \in \Gamma \backslash (\rho_k, \sigma_k)$ such that $\rho = \rho_k$ and there does not exist a $(\rho, \sigma) \in \Gamma \backslash (\rho_k, \sigma_k)$ such that $\sigma = \sigma_k$.

\item If $(\rho_1, \sigma_1) \in H_{set}$ and $(\rho_2, \sigma_2) \in V_{set}$, then $|\Gamma| = 2k+1$, where $k = 1,2,3, \cdots$. Otherwise, we have $|\Gamma| = 2k$, where $k = 1,2,3, \cdots$. 
\end{enumerate}
\end{theorem}

\subsection{The Detection Algorithm}
We next introduce an algorithm which checks whether the invariant measure of a given random walk is a sum of geometric terms. We call this algorithm the Detection Algorithm. The Detection Algorithm is based on the construction of several pairwise-coupled sets. In particular, we construct one such set for each of the elements in $H_{set} \cup V_{set}$. If such a pairwise-coupled set contains another geometric term from $H_{set} \cup V_{set}$ before it goes outside of the unit square, then the invariant measure of the random walk may be a sum of geometric terms. Next, the algorithm checks whether the geometric terms from the pairwise-coupled set are coupled in a correct manner, if Condition 2 in Theorem~\ref{thm:oddD} is satisfied. Notice that because of Theorem~\ref{thm:oddD}, if such a pairwise-coupled set exists, then it must be unique.

\noindent\rule{11.7cm}{0.4pt} \\
\textbf{The Detection Algorithm}: Check whether the invariant measure of a random walk $R$ is a sum of geometric terms.\\
\noindent\rule{11.7cm}{0.4pt} \\
\noindent Input: The random walk $R$: $p_{s,t}$, $h_s$, $v_t$ for $s,t \in \{-1,0,1\}$.\\
{\bf Step 1}: Compute $H_{set}$ and $V_{set}$.\\
{\bf Step 2}: Let $(\rho_1, \sigma_1) \in H_{set}$. We construct a set $\Gamma^{H}_1$ as follows: we have $\rho_{2k} \neq \rho_{2k-1}$, $\sigma_{2k} = \sigma_{2k-1}$ and $\rho_{2k+1} = \rho_{2k}$, $\sigma_{2k+1} \neq \sigma_{2k}$ for $k = 1,2,3, \cdots$. We continue this procedure until we have $(\rho_n, \sigma_n) \in (0,1)^2$ and $(\rho_{n+1}, \sigma_{n+1}) \notin (0,1)^2$. We denote this pairwise-coupled set with $n$ elements by $\Gamma^{H}_1$ where $\Gamma^{H}_1 = \{(\rho_1, \sigma_1), (\rho_2, \sigma_2), \cdots, (\rho_n, \sigma_n)\}$. We repeat this procedure for other elements from $H_{set}$. The resulting pairwise-coupled sets are denoted by $\Gamma^{H}_a$ where $a \leq 3$ and denotes the index of the element in $H_{set}$.\\
{\bf Step 3}: Let $(\rho_1, \sigma_1) \in V_{set}$. We construct a set $\Gamma^{V}_1$ as follows: we have $\rho_{2k} = \rho_{2k-1}$, $\sigma_{2k} \neq \sigma_{2k-1}$ and $\rho_{2k+1} \neq \rho_{2k}$, $\sigma_{2k+1} = \sigma_{2k}$ for $k = 1,2,3, \cdots$. 
We continue this procedure until we have $(\rho_n, \sigma_n) \in (0,1)^2$ and $(\rho_{n+1}, \sigma_{n+1}) \notin (0,1)^2$. We denote this pairwise-coupled set with $n$ elements by $\Gamma^{V}_1$ where $\Gamma^{V}_1 = \{(\rho_1, \sigma_1), (\rho_2, \sigma_2), \cdots, (\rho_n, \sigma_n)\}$. We repeat this procedure for other elements from $V_{set}$. The resulting pairwise-coupled sets are denoted by $\Gamma^{V}_b$ where $b \leq 3$ and denotes the index of the element in $V_{set}$.\\
{\bf Step 4}: Check whether the geometric terms from the pairwise-coupled set are coupled in a correct manner.
\begin{enumerate}
\item Consider $\Gamma^{H}_1$. If there exists $k_0 \in \mathbb{N}_0$ such that the element $(\rho_{2k_0}, \sigma_{2k_0}) \in H_{set}$, then the invariant measure of $R$ is a sum of geometric terms. Moreover, the invariant measure is induced by the set $\{(\rho_1, \sigma_1), \dots, (\rho_{2k_0}, \sigma_{2k_0})\} \subset \Gamma^{H}_1$. We repeat this procedure for all $\Gamma^H_a$ where $a \leq 3$.

\item Consider $\Gamma^{H}_1$. If there exists $k_0 \in \mathbb{N}_0$ such that the element $(\rho_{2k_0 + 1}, \sigma_{2k_0 + 1}) \in V_{set}$, then the invariant measure of $R$ is a sum of geometric terms. Moreover, the invariant measure is induced by $\{(\rho_1, \sigma_1), \dots, (\rho_{2k_0 + 1}, \sigma_{2k_0 + 1})\} \subset \Gamma^{H}_1$. We repeat this procedure for all $\Gamma^H_a$ where $a \leq 3$.

\item Consider $\Gamma^{V}_1$. If there exists $k_0 \in \mathbb{N}_0$ such that the element $(\rho_{2k_0}, \sigma_{2k_0}) \in V_{set}$, then the invariant measure of $R$ is a sum of geometric terms. Moreover, the invariant measure is induced by $\{(\rho_1, \sigma_1),  \dots, (\rho_{2k_0}, \sigma_{2k_0})\} \subset \Gamma^{V}_1$. We repeat this procedure for all $\Gamma^V_b$ where $b \leq 3$.

\item Consider $\Gamma^{V}_1$. If there exists $k_0 \in \mathbb{N}_0$ such that the element $(\rho_{2k_0 + 1}, \sigma_{2k_0 + 1}) \in H_{set}$, then the invariant measure of $R$ is a sum of geometric terms. Moreover, the invariant measure is induced by $\{(\rho_1, \sigma_1),  \dots, (\rho_{2k_0 + 1}, \sigma_{2k_0 + 1})\} \subset \Gamma^{V}_1$. We repeat this procedure for all $\Gamma^V_b$ where $b \leq 3$.
\end{enumerate}
{\bf Step 5}:
If none of the conditions in Step $4$ holds, then the invariant measure of $R$ is not a sum of geometric terms.\\
\noindent\rule{11.7cm}{0.4pt}

We now apply the Detection Algorithm to two examples.

\begin{example}{\label{ex:R1D}}
We have $p_{1,0} = 0.05$, $p_{-1,1} = 0.15$, $p_{0,-1} = 0.15$, $p_{0,0} = 0.65$, $h_1 = 0.15$, $h_0 = 0.7$, $v_1 = 0.0929$, $v_{-1} = 0.15$, $v_{0} = 0.7071$. The other transition probabilities are zero.
\end{example}

We see from Figure~\ref{fig:ALGEXAD} that there exists a pairwise-coupled set with $3$ geometric terms satisfying the criterion in the Detection Algorithm. Hence, the invariant measure of Example~\ref{ex:R1D} is a sum of $3$ geometric terms.


\begin{figure}
\begin{center}
\subfigure[]{\includegraphics[scale=1]{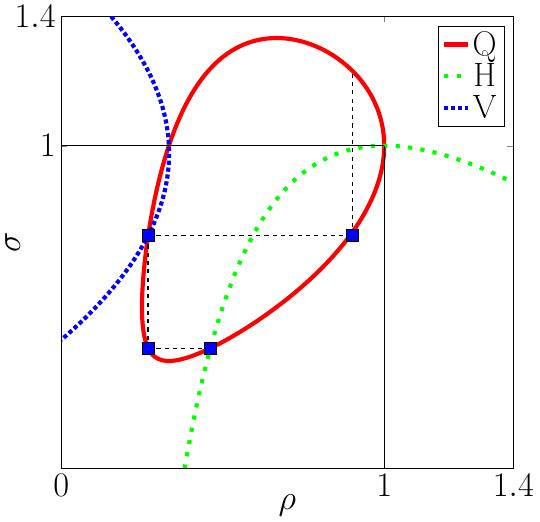}{\label{fig:4aD}}} \:\:\:\:\:\:
\subfigure[]{\includegraphics[scale=1]{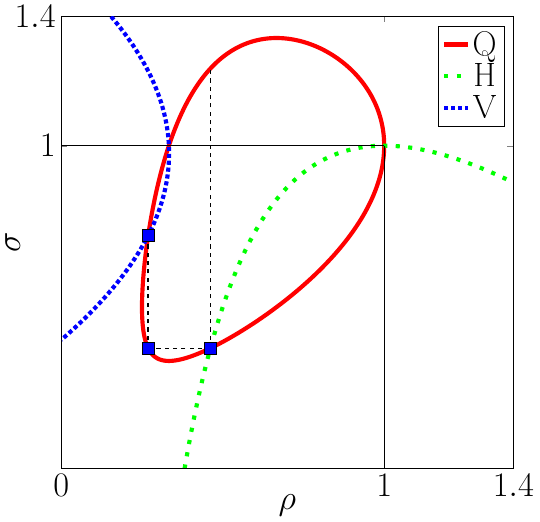}{\label{fig:4bD}}}
\end{center}
\caption{Apply the Detection Algorithm to Example~\ref{ex:R1D}. The geometric terms are denoted by the squares.~\subref{fig:4aD} $\Gamma^{H}_1$.~\subref{fig:4bD} $\Gamma^{V}_1$. \label{fig:ALGEXAD}}
\end{figure}

\begin{example}{\label{ex:R2D}}
We have $p_{1,0} = 0.05$, $p_{0,1} = 0.05$, $p_{-1,1} = 0.2$, $p_{-1,0} = 0.2$, $p_{0,-1} = 0.2$, $p_{1,-1} = 0.2$, $p_{0,0} = 0.1$, $h_1 = 0.5$, $h_{-1} = 0.1$, $h_0 = 0.15$, $v_1 = 0.1$, $v_{-1} = 0.06$, $v_0 = 0.59$. The other transition probabilities are zero.
\end{example}


\begin{figure}
\begin{center}
\subfigure[]{\includegraphics[scale=1]{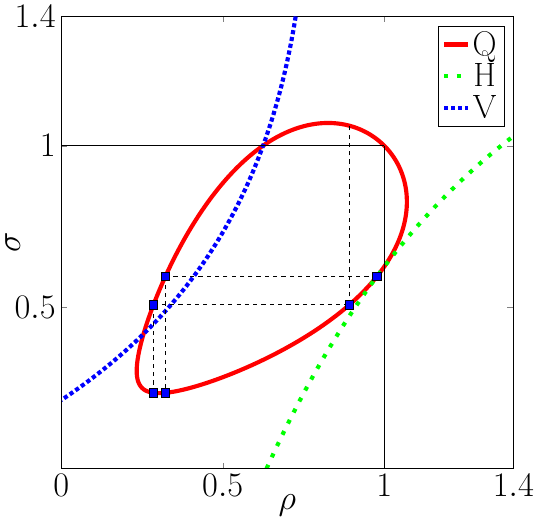}{\label{fig:EB7D}}} \:\:\:\:\:\:
\subfigure[]{\includegraphics[scale=1]{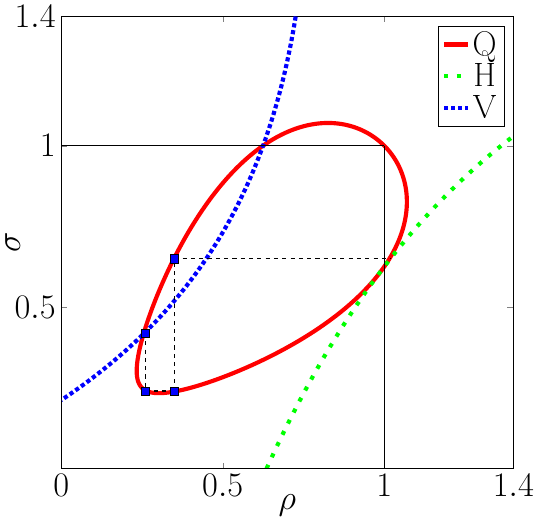}{\label{fig:EBVD}}} 
\end{center}
\caption{Apply the Detection Algorithm to Example~\ref{ex:R2D}. The geometric terms are denoted by the squares.~\subref{fig:EB7D} $\Gamma^{H}_1$.~\subref{fig:EBVD} $\Gamma^{V}_1$. \label{fig:ALGEXBextraextraD}}
\end{figure}

In Figures~\ref{fig:EB7D} and~\ref{fig:EBVD}, we construct the pairwise-coupled sets, $\Gamma^H_1$ and $\Gamma^V_1$, starting from the elements in $H_{set} = \{(\rho_h, \sigma_h)\}$ and $V_{set} = \{(\rho_v, \sigma_v)\}$, respectively, until they go outside of the unit square. Notice that $\Gamma^H_1 \cap (H_{set} \cup V_{set} \backslash (\rho_h, \sigma_h)) = \emptyset$ and $\Gamma^V_1 \cap (H_{set} \cup V_{set} \backslash (\rho_v, \sigma_v)) = \emptyset$. Hence, the invariant measure of Example~\ref{ex:R2D} cannot be a sum of geometric terms.

\subsection{Running time of the Detection Algorithm}

In this section, we show that the Detection Algorithm has a finite running time. More precisely, we provide an upper bound on the number of terms in the pairwise-coupled sets that are constructed in Steps $1$ and $2$ of the Detection Algorithm. In particular, we show that this construction provides a geometric term outside the unit square in a finite number of steps.

We first introduce the notion of branch points of $Q$. A point $(x_0, y_0) \in Q$ that satisfies $\Delta_y(x_0) = 0$, where
\begin{align}
\Delta_y(x) = \left(\sum_{s = -1}^{1} x^{-s + 1} p_{s,0} - x\right)^2 - 4\left(\sum_{s = -1}^{1} x^{-s + 1} p_{s,-1}\right)\left(\sum_{s = -1}^{1} x^{-s + 1} p_{s,1}\right), {\label{eq:YdeltaD}}
\end{align}
is called a horizontal branch point of $Q$. Since the algebraic curve $Q$ has a unique connected component in $[0,\infty)^2$ (see Lemma~7,~\cite{chen2013necessary}), it has two horizontal branch points in $[0,\infty)^2$, denoted by $(x_b, y_b)$, $(x_t, y_t)$ with $y_t \geq y_b$. In Appendix~\ref{app:proofD} we provide more details on the branch points of $Q$ as well as a proof of the next result.
\begin{theorem}{\label{thm:maximalstepsD}}
Consider a random walk $R$. For any pairwise-coupled set $\Gamma \subset Q$, if $|\Gamma| > M(R)$ with $M(R) = \frac{6}{\min(D_1, D_2)} + 4$, $D_1 = \frac{\Delta_y(x_b)}{\sum_{s = -1} ^{1} p_{s,-1} x_t^{1 - s}}$, $D_2 = \frac{\Delta_y(x_t)}{\sum_{s = -1} ^{1} p_{s,-1} x_t^{1 - s}}$, then there exists a $(\rho, \sigma) \in \Gamma$ such that $\rho > 1$ or $\sigma > 1$. Moreover, when $p_{1,0} + p_{1,1} + p_{0,1} \neq 0$, we have $D_1 > 0$ and $D_2 > 0$. Hence, $M(R) < \infty$.
\end{theorem}
Theorem~\ref{thm:maximalstepsD} guarantees that the Detection Algorithm stops in finite time, since the pairwise-coupled set will go outside of the unit square with a bounded number of steps.

\subsection{Construction of the invariant measure}
If a random walk has an invariant measure that is a sum of geometric terms, then the Detection Algorithm will provide the set $\Gamma$ that induces this measure. It remains to construct the weighting coefficients in the linear combination of geometric terms. We next explain how to find the coefficients in the induced measure if the invariant measure of the random walk is a sum of geometric terms.

We will use Lemma~6 from~\cite{chen2012invariant} to determine the coefficients in the induced measure. We find it convenient to repeat it here.
\begin{lemma}{\label{lem:6D}}
Consider the random walk $R$ and a measure $m$ induced by $\Gamma$. Then $m$ is the
invariant measure of $R$ if and only if $B^h(\rho,\sigma) = 0$ and $B^v(\rho,\sigma) = 0$ for all $(\rho,\sigma)\in\Gamma$, where
\begin{align}{\label{eq:HdecomposeD}}
B^h(\tilde\rho,\tilde\sigma) = \sum_{(\rho, \sigma)\in\Gamma: \rho=\tilde\rho}\alpha(\rho, \sigma)\left[\sum_{s=-1}^1 \big(\rho^{1-s} h_s+\rho^{1-s}\sigma p_{s,-1}\big) - \rho\right],  
\end{align}
\begin{align}{\label{eq:VdecomposeD}}
B^v(\tilde\rho,\tilde\sigma) &= \sum_{(\rho, \sigma)\in\Gamma: \sigma=\tilde\sigma}\alpha(\rho, \sigma)\left[\sum_{t=-1}^1 \big(\sigma^{1-t} v_t+\rho \sigma^{1-t} p_{-1,t}\big) - \sigma\right]. 
\end{align}
\end{lemma}
Lemma~\ref{lem:6D} states that every two coefficients of the geometric terms with the same horizontal or vertical coordinates must satisfy a linear relationship.

The construction of the coefficients is as follows. We first fix the weighting coefficient for one of the terms in $\Gamma$ to an arbitrary value. Since this term is coupled to other terms in $\Gamma$, we can now compute the weighting coefficient for these terms using Lemma~\ref{lem:6D}. Following the same reasoning, we obtain values for all coefficients. Finally, we rescale all coefficients in order to ensure $\sum_{i = 0}^{\infty} \sum_{j = 0}^{\infty} m(i,j) = 1$.
 
We provide an example for the case that 
\begin{equation*}
\Gamma = \{(\rho_1, \sigma_1), (\rho_2, \sigma_2), (\rho_3, \sigma_3), \dots\},
\end{equation*} 
where $\rho_1 \neq \rho_2, \sigma_1 = \sigma_2$, $\rho_2 = \rho_3$, $\sigma_2 \neq \sigma_3$, $\dots$. We first fix $\alpha(\rho_1, \sigma_1)=1$. Next, we compute $\alpha(\rho_k, \sigma_k)$ consecutively, for $k = 2,3,4,\dots$. The value of the even terms $\alpha(\rho_{2\ell},\sigma_{2\ell})$, where $\ell = 1,2,3,\dots$ is given by: 
\begin{equation*}
\alpha(\rho_{2\ell}, \sigma_{2\ell}) = - \frac{W_{2\ell - 1}}{W_{2\ell}}  \alpha(\rho_{2\ell-1},\sigma_{2\ell-1}),
\end{equation*}
where
\begin{equation*}
W_k = (1 - \frac{1}{\sigma_k}) v_1 + (1 - \sigma_k)v_{-1} + \sum_{t = -1}^{1} p_{1,t} - \rho_k\left(\sum_{t = -1}^{1} \sigma_k^{-t}p_{-1,t}\right).
\end{equation*}
The value of the odd terms $\alpha(\rho_{2\ell+1}, \sigma_{2\ell+1})$, where $\ell = 1,2,3,\dots$ is given by:
\begin{equation*}
\alpha(\rho_{2\ell+1}, \sigma_{2\ell+1}) = - \frac{T_{2\ell}}{T_{2\ell+1}}  \alpha(\rho_{2\ell},\sigma_{2\ell}), 
\end{equation*}
where
\begin{equation*}
T_k = (1 - \frac{1}{\rho_k}) h_1 + (1 - \rho_k)h_{-1} + \sum_{s = -1}^{1} p_{s,1} - \sigma_k\left(\sum_{s = -1}^{1} \rho_k^{-s}p_{s,-1}\right).
\end{equation*}
Finally, all coefficients are rescaled to obtain a probability measure.

\section{Approximation analysis}{\label{sec:LPapproximationD}}

In this section, we provide an approximation scheme to establish upper and lower bounds for performance measures of a random walk for which the invariant measure is unknown. This scheme is similar to that developed in~\cite{goseling2014linear} in the sense that a linear program is developed to approximate the performance measures. We will give an overview of this approach below. The main difference in the current work is that we enlarge the candidate set of perturbed random walks that can be used in the approximation scheme. In particular, our approximation is based on a perturbed random walk of which the invariant measure is a sum of geometric terms. In addition to presenting the scheme itself in Section~\ref{subsec:ASD}, we show how such a perturbed random walk can be constructed in Section~\ref{subsec:perturbRWD}.


\subsection{Approximation scheme}{\label{subsec:ASD}}

Recall from Section~\ref{sec:modelD} that our approximation analysis provides, for a given performance measure, upper and lower bounds for $\mathcal{F} = \sum_{(i,j) \in S} m(i,j) F(i,j).$ Moreover, we consider the case that $m$ is unknown. In particular, $m$ is not a sum of geometric terms. The upper and lower bound on $\mathcal{F}$ that our approximation provide are expressed in terms of the invariant measure $\bar m: S\to[0,\infty)$ of another random walk $\bar R$, which we will refer to as the perturbed random walk. The perturbed random walk that we will consider has $\bar{m} = \sum_{(\rho, \sigma) \in \Gamma} \alpha(\rho, \sigma) \rho^i \sigma^j$ for some $\Gamma$. The transition probabilities of $\bar{R}$ are denoted by $\bar{p}_{s,t}$ for $s,t \in \{-1,0,1\}$. Moreover, we use $q_{s,t}$ where $s,t \in \{-1,0,1\}$ to denote the difference between the transition probabilities in $R$ and the corresponding transition probabilities in $\bar{R}$. In Section~\ref{subsec:perturbRWD} we provide a construction of $\bar{R}$. In the remainder of this section we present our approximation scheme based on the assumption that $\bar{R}$ is already known.

We interpret $F$ as a reward function, where $F^t(i,j)$ is the one step reward if the random walk is in state $(i,j)$. We denote by $F^t(i,j)$ the expected cumulative reward at time $t$ if the random walk starts from state $(i,j)$ at time $0$, \ie

\begin{equation*}
F^t(i,j) =
\begin{cases}
0, \quad &\text{if } t = 0,\\
F(i,j) + \sum_{u,v \in \{-1,0,1\}} p_{u,v} F^{t-1} (i+u,j+v), \quad &\text{if } t > 0.
\end{cases}
\end{equation*}

The next result from~\cite{vandijk11inbook} provides bounds on the approximation errors for $\mathcal{F}$. 


\begin{theorem}[\cite{vandijk11inbook}]{\label{thm:boundsD}}
Let $\bar{F}:S \rightarrow [0,\infty)$ and $G:S \rightarrow [0,\infty)$ satisfy
\begin{equation}{\label{eq:conditionsD}}
|\bar{F}(i,j) - F(i,j) + \sum_{u,v \in \{-1,0,1\}} q_{u,v} (F^t(i+u, j+v) - F^t(i,j))| \leq G(i,j),
\end{equation}
for all $(i,j) \in S$ and $t \geq 0$. Then,
\begin{equation*}
\sum_{(i,j) \in S} [\bar{F}(i,j) - G(i,j)]\bar{m}(i,j) \leq \mathcal{F} \leq \sum_{(i,j) \in S} [\bar{F}(i,j) + G(i,j)]\bar{m}(i,j).
\end{equation*}
\end{theorem}
Based on Theorem~\ref{thm:boundsD}, we develop a linear program similar to that in~\cite{goseling2014linear} to approximate $\mathcal{F}$. This linear program provides a set of sufficient constraints for~\eqref{eq:conditionsD} and an objective function that optimizes the upper (or lower) bound on $\mathcal{F}$.  In the linear program, we consider $\bar{F}$ and $G$ as variables and $q_{u,v}$, $F^t$ and $\bar{m}$ as parameters. In~\cite{goseling2014linear}, the invariant measure of the perturbed random walk is only allowed to be of product-form. The difference with the current work is that in our approximation scheme the invariant measure of the perturbed random walk can also be a sum of geometric terms. This difference will only affect the objective function of the linear program.

The constraints developed in~\cite{goseling2014linear} are obtained from Theorem~\ref{thm:boundsD}. Without additional precautions the number of constraints that is obtained is not finite because the state space $S$ contains infinitely many states and the time horizon is also infinite. In order to have a finite linear program with finitely many constraints the variables and the parameters in the linear program are constrained to be component-wise linear functions, \ie similar to how $F(i,j)$ is defined in Section~\ref{sec:modelD}. The rationale behind this is that non-negativity constraints on (linear combinations of) such functions can be expressed in a finite number of constraints. As second step towards a finite linear program, $F^t(i+u, j+v) - F^t(i,j)$ where $u,v \in \{-1,0,1\}$ is uniformly bound over $t$. In this manner finitely many constraints in finitely many variables are obtained. We refer the reader to~\cite{goseling2014linear} for details. For the remainder of this paper we capture the required results from~\cite{goseling2014linear}  in the following theorem which captures sufficient conditions for~\eqref{eq:conditionsD} in terms of a polytope $\mathcal{P}$.
\begin{theorem}[\cite{goseling2014linear}]{\label{thm:boundsDD}} 
Let $\bar F$ and $G$ be component-wise linear functions. If $(\bar F, G)\in\mathcal{P}$ then
\begin{equation*}
 \sum_{(i,j)\in S}\left[\bar F(i,j) - G(i,j)\right]\bar m(i,j)\ \leq\ \mathcal{F}\ \leq\ \sum_{(i,j)\in S}\left[\bar F(i,j) + G(i,j)\right]\bar m(i,j).
\end{equation*}
Moreover, $\mathcal{P}$ can be represented with a finite number of constraints that are linear in the coefficients that define $\bar F$ and $G$.
\end{theorem}

The difference with~\cite{goseling2014linear} is that instead of only using a perturbed random walk of which the invariant measure is of product-form, we are also allowed to use a perturbed random walk of which the invariant measure is a sum of geometric terms as well. Hence, we are able to find the upper and lower bounds for $\mathcal{F}$ based on a richer set of perturbed random walks with closed-form invariant measures $\bar{m}$. The linear programs that provide $F_{up}$ and $F_{low}$ are
\begin{align*}
F_{up} &= \min \left\{\sum_{(i,j) \in S} [\bar{F}(i,j) + G(i,j)] \bar{m}(i,j)\ \middle|\ (\bar F, G)\in\mathcal{P} \right\}, \\
F_{low}&= \max \left\{\sum_{(i,j) \in S} [\bar{F}(i,j) - G(i,j)] \bar{m}(i,j)\ \middle|\ (\bar F, G)\in\mathcal{P} \right\}.
\end{align*}

Notice that the measure $\bar{m}$ is a parameter in the above linear programs. The transition rates in $\bar R$, moreover, affect the constraints in $\mathcal{P}$. We will demonstrate in Section~\ref{sec:examplesD} that the choice of $\bar R$ and $\bar m$ can significantly affect the bounds $F_{up}$ and $F_{low}$. The best upper and lower bounds might be achieved by using different perturbed random walks. In the next section, we explain how to find such a perturbed random walk $\bar{R}$.

\subsection{Perturbed random walk $\bar{R}$}{\label{subsec:perturbRWD}}
In this section, we first show how to find a perturbed random walk $\bar{R}$ such that the invariant measure of $\bar{R}$ is of a given product-form, \ie $\bar{m}(i,j) = \rho^i \sigma^j$. Then, we show how to find the perturbed random walk $\bar{R}$ of which the invariant measure is a sum of geometric terms. We restrict our attention to the case where for the perturbed random walk $\bar{R}$, only the transitions along the boundaries of the state space $S$, which are denoted by $\bar{h}_1, \bar{h}_{-1}, \bar{h}_{0}$ and $\bar{v}_1, \bar{v}_{-1}, \bar{v}_0$, are different from that in the original random walk $R$, see Figure~\ref{fig:PrwD}. 

\begin{figure}
\hfill
\includegraphics{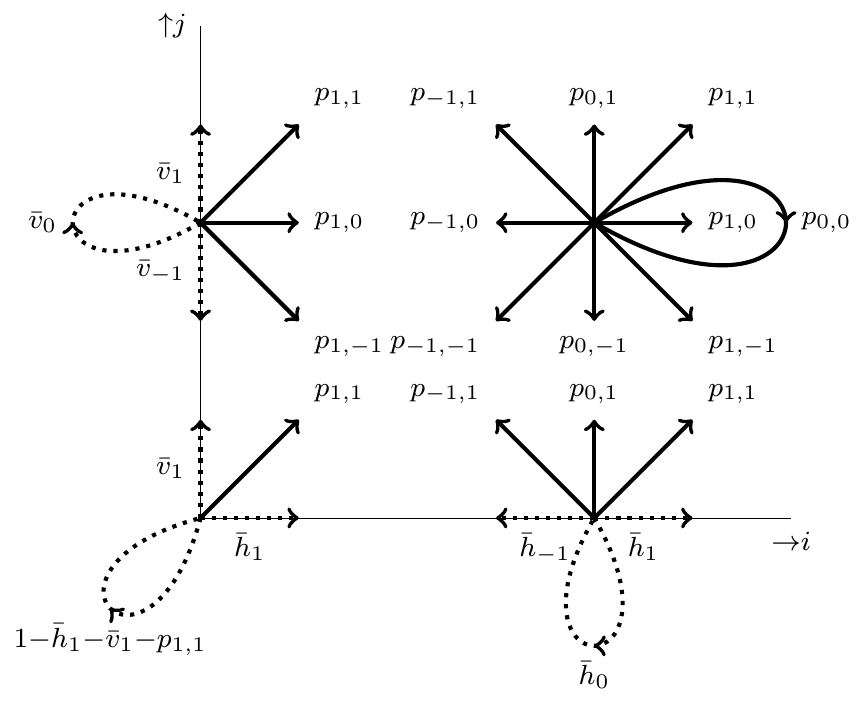}
\hfill{}
\caption{Perturbed random walk in the quarter-plane. \label{fig:PrwD}}
\end{figure}

The construction of $\bar{h}_1, \bar{h}_{-1}, \bar{h}_{0}$ and $\bar{v}_1, \bar{v}_{-1}, \bar{v}_0$ will take place in three phases. Before providing details we provide an overview of all phases. In the first phase we will let $\bar h_0= \bar v_0 = 0$ and find non-negative values of $\bar{h}_1$, $\bar{h}_{-1}$, $\bar{v}_1$, $\bar{v}_{-1}$ that ensure that all balance equations are satisfied. We allow these values to be larger than one and we do not require that outgoing transitions sum to one. In the second phase we scale the values of all transitions probabilities in $\bar R$ and the values of $\bar{h}_1$, $\bar{h}_{-1}$, $\bar{v}_1$, $\bar{v}_{-1}$ such that the resulting values are properly normalized transition probabilities. In the third phase we scale the transition probabilities of $R$ with the same constant.

Next, we define the scaling operation more precisely in terms of the notion of a $C$-rescaled random walk. Since the values of $\bar{h}_1$, $\bar{h}_{-1}$, $\bar{v}_1$, $\bar{v}_{-1}$ that are obtained in phase 1 are not necessarily transition probabilities we denote them by $H_1$, $H_{-1}$, $V_1$, $V_{-1}$ to avoid possible confusion. The resulting system might not be a random walk and will be denoted by $\mathcal{R}$.


%
\begin{definition}[$C$-rescaled random walk]{\label{def:rescaleD}}
Consider $\mathcal{R}$ with $p_{s,t}$ where $s,t \in \{-1,0,1\}$ in the interior and $H_1$, $H_{-1}$, $V_1$, $V_{-1}$ for the boundaries. Let $C > 1$. The random walk $\tilde{R}$ is called the $C$-rescaled random walk of $\mathcal{R}$ if the transition probabilities of $\tilde{R}$ are $\tilde{p}_{s,t} = \frac{p_{s,t}}{C}$ for $(s,t) \neq (0,0)$ and $\tilde{p}_{0,0} = 1 - \sum_{(s,t) \neq (0,0)} \tilde{p}_{s,t}$. Moreover, the boundary transition probabilities are $\tilde{h}_1 = \frac{H_1}{C}$, $\tilde{h}_{-1} = \frac{H_{-1}}{C}$, $\tilde{h}_{0} = 1 - \tilde{h}_1 - \tilde{h}_{-1} - \sum_{s = -1}^{1} \tilde{p}_{s,1}$,  $\tilde{v}_1 = \frac{V_1}{C}$, $\tilde{v}_{-1} = \frac{V_{-1}}{C}$ and $\tilde{v}_{0} = 1 - \tilde{v}_1 - \tilde{v}_{-1} - \sum_{t = -1}^{1} \tilde{p}_{1,t}$.
\end{definition}
Clearly, if $\bar m(i,j) = \rho^i \sigma^j$ satisfies all balance equations for all states in $\mathcal{R}$, then $\bar m(i,j)$ satisfies all balance equations for all states in a $C$-rescaled random walk of $\mathcal{R}$ as well. Therefore, scaling the original random walk $R$ with the same constant $C$, will not affect its invariant measure. Moreover, we can now apply the approximation scheme from the previous section by comparing the two rescaled random walks, since they will have the same transition probabilities in the interior of the state space.

If the given measure is of product-form, then we find the new boundary probabilities in the perturbed random walk such that the algebraic curve $H$ and $V$ will cross this point which leads to the product-form invariant measure. If the given measure is a sum of an odd number of geometric terms which form a pairwise-coupled set, then we find the new boundary probabilities in the perturbed random walk such that the algebraic curves $H$ and $V$ will cross two specific points from this pairwise-coupled set.

\subsubsection{$\bar{R}$ with a product-form invariant measure}

The first step is to construct the scalars $H_1, H_{-1}$ and $V_1, V_{-1}$ which can be greater than $1$. Notice that, if $\bar{m}(i,j) = \rho^i \sigma^j$ satisfies the balance equations for all states from $S$, the scalars $H_1, H_{-1}$ and $V_1, V_{-1}$ must satisfy the horizontal and vertical balance equations in~\eqref{eq:horizontalD} and~\eqref{eq:verticalD}, respectively. Inserting $m(i,j) = \rho^i \sigma^j$ into equation~\eqref{eq:horizontalD} gives
\begin{equation}{\label{eq:shorizontalD}}
(1 - \frac{1}{\rho}) H_1 + (1 - \rho) H_{-1} = \sum_{s = -1}^{1} \rho^{-s} \sigma p_{s,-1} - \sum_{s = -1}^{1} p_{s,1}.
\end{equation}
Equation~\eqref{eq:shorizontalD} is linear in the two unknowns $H_1, H_{-1}$. The non-negative $H_1, H_{-1}$ exist because the slope of equation~\eqref{eq:shorizontalD} is positive. Similarly, non-negative $V_1$ and $V_{-1}$ can also be found.

The next step is to start rescaling. Therefore, we determine $c_h$ and $c_v$ as follows,
\begin{equation*}
c_h = H_1 + H_{-1} + \sum_{s = -1}^{1}p_{s,1} \quad \text{and} \quad c_v = V_1 + V_{-1} + \sum_{t = -1}^{1}p_{1,t}.
\end{equation*}
We take $C = \max\{c_h, c_v\}$. 

If $C \leq 1$, then the random walk with $\bar{h}_1 = H_1, \bar{h}_{-1} = H_{-1}, \bar{v}_1 = V_1, \bar{v}_{-1} = V_{-1}$ as boundary transition probabilities is the perturbed random walk $\bar{R}$ with invariant measure $\bar{m}$. 

If $C > 1$, we do not consider the original random walk $R$ anymore. We consider the $C$-rescaled random walk of $R$, which is denoted by $\tilde{R}$, to be the input random walk for the approximation scheme. Moreover, the random walk $\bar{R}$ with $\bar{h}_1 = \frac{H_1}{C}, \bar{h}_{-1} = \frac{H_{-1}}{C}, \bar{v}_1 = \frac{V_1}{C}, \bar{v}_{-1} = \frac{V_{-1}}{C}$ as boundary transition probabilities is the perturbed random walk of $\tilde{R}$, instead of $R$.

\subsubsection{$\bar{R}$ with a sum of geometric terms as invariant measure}
Similarly, we find the perturbed random walk $\bar{R}$ of which the invariant measure is $\bar{m} = \sum_{(\rho, \sigma) \in \Gamma} \alpha(\rho, \sigma) \rho^i \sigma^j$ where $|\Gamma| = 2k+1$, with $k = 1,2,3,\dots$.

Using Theorem~\ref{thm:oddD}, we find $(\rho_1, \sigma_1)$ to be the geometric term from the set $\Gamma$ which does not share the horizontal coordinate with any other geometric terms from set $\Gamma$. Similarly, we find $(\rho_2, \sigma_2)$ to be the geometric term from the set $\Gamma$ which does not share the vertical coordinate with any other geometric terms from set $\Gamma$.

Instead of looking for the scalars $H_1$, $H_{-1}$, $V_1$ and $V_{-1}$ which satisfy the horizontal and vertical balance equation for the product-form $\bar{m}(i,j) = \rho^i \sigma^j$, we find $H_1$ and $H_{-1}$, which satisfy the horizontal balance for the geometric measure $\rho_1^i \sigma_1^j$, and $V_1$ and $V_{-1}$, which satisfy the vertical balance for the geometric measure $\rho_2^i \sigma_2^j$ (see Lemma~\ref{lem:6D}). Again, we compute $c_h$ and $c_v$. If $C\leq 1$, we find $\bar{R}$ directly. If $C > 1$, we consider the $C$-rescaled random walk of $R$, $\tilde{R}$, as the input random walk for our approximation scheme. Moreover, we find the perturbed random $\bar{R}$ for $\tilde{R}$ similar to the case when the invariant measure of $\bar{R}$ is of product-form.
\section{Numerical illustrations}{\label{sec:examplesD}}
In this section, we apply the Detection Algorithm and the approximation scheme developed in Section~\ref{sec:LPapproximationD} to several random walks. For any given random walk, we provide the explicit form of the invariant measure if it is a sum of geometric terms. Otherwise, we provide error bounds for the performance measures. 

We show that the bounds for the performance measures will be improved when using a perturbed random walk of which the invariant measures is a sum of geometric terms instead of a perturbed random walk of which the invariant measure is of product-form as in~\cite{goseling2014linear}. 

In particular, we are interested in following performance measures of a random walk in the quarter-plane,\\
$\mathcal{F}_1$: the average number of jobs in the first dimension,\\
$\mathcal{F}_2$: the probability that the system is empty.\\
Notice that the function $F(i,j)$ used to determine $\mathcal{F}$ is component-wise linear. It can the readily verified that the performance measure $\mathcal{F}$ is $\mathcal{F}_1$ if and only if we assign the following values to the coefficients: $f_{1,1} = 1$ and $f_{4,1} = 1$ and others $0$. Similarly, the performance measure $\mathcal{F}$ is $\mathcal{F}_2$ if and only if we assign the following values to the coefficients: $f_{3,0} = 1$ and others $0$. 

We use $[\mathcal{F}^1_1]_{up/low}$, $[\mathcal{F}^1_2]_{up/low}$ to denote the upper and lower bounds for the performance measures obtained via a perturbed random walk of which the invariant measure is of product-form. We use $[\mathcal{F}^3_1]_{up/low}$, $[\mathcal{F}^3_2]_{up/low}$ to denote the upper and lower bounds for the performance measures obtained via a perturbed random walk of which the invariant measure is a sum of $3$ geometric terms.

The next example we consider here has also been mentioned in~\cite{chen2012invariant}. 
\begin{example}{\label{ex:oneD}}
We have $p_{1,0} = 0.05$, $p_{0,1} = 0.05$, $p_{-1,1} = 0.2$, $p_{-1,0} = 0.2$, $p_{0,-1} = 0.2$, $p_{1,-1} = 0.2$, $p_{0,0} = 0.1$ and $h_1=0.5$, $h_{-1} = 0.1$, $h_0 = 0.15$, $v_1 = 0.113$, $v_{-1}=0.06$, $v_0 = 0.577$. The other transition probabilities are zero. 
\end{example}
In Figure~\ref{fig:crbq1D} all non-zero transition probabilities, except those for the transitions from a state to itself, are illustrated.

\begin{figure}
\begin{center}
\subfigure[]{\includegraphics[scale=1]{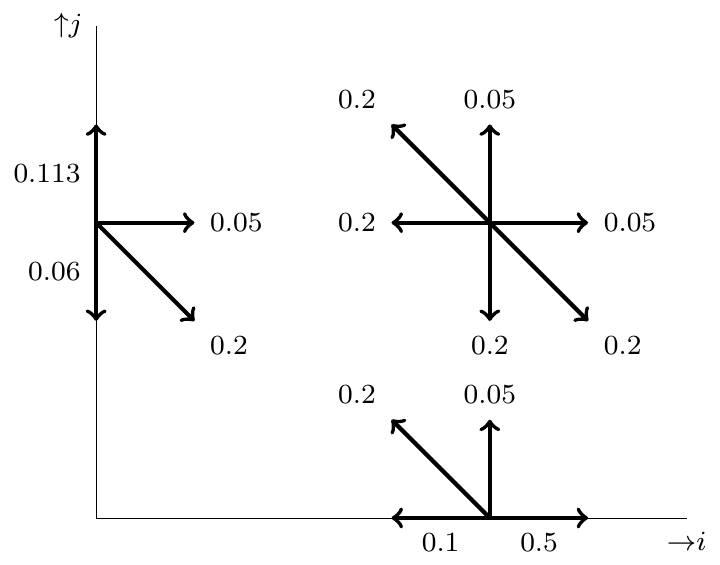}{\label{fig:crbq1D}}}
 \\
\subfigure[]{\includegraphics[scale=1]{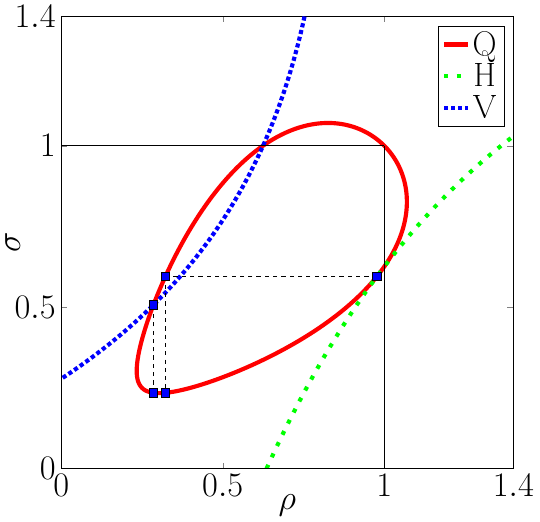}{\label{fig:crbq2D}}}
\end{center}
\caption{Example~\ref{ex:oneD}.~\subref{fig:crbq1D} Transition diagram.~\subref{fig:crbq2D} Algebraic curves $Q$, $H$ and $V$. The geometric terms contributed to the invariant measure are denoted by the squares.\label{fig:crbqD}}
\end{figure}

Using the Detection Algorithm, we find the invariant measure of this random walk is 
\begin{equation*}
m(i,j) = \sum_{k=1}^5 \alpha_k\rho_k^i\sigma_k^j, 
\end{equation*}
where the geometric terms are the solid squares in Figure~\ref{fig:crbq2D} and the coefficients are $\alpha_1 = 0.0088$, $\alpha_2=0.1180$, $\alpha_3=-0.1557$, $\alpha_4 = 0.1718$, $\alpha_5 = -0.1414$. Therefore, we are able to compute the performance measures directly from $m(i,j)$,
\begin{align*}
\mathcal{F}_1 &= \sum_{k = 1}^5 \sum_{i = 0}^{\infty} \sum_{j = 0}^{\infty} \alpha_k i \rho_k^i \sigma_k^j = \sum_{k = 1}^{5} \alpha_k \frac{\rho_k}{(1 - \rho_k)^2} \frac{1}{1 - \sigma_k} = 41.2062. \\
\mathcal{F}_2 &= \sum_{k = 1}^{5} \alpha_k = 0.0015.
\end{align*}

\begin{example}{\label{ex:twoD}}
We have $p_{1,0} = 0.1$, $p_{0,1} = 0.1$, $p_{-1,1} = 0.1$, $p_{-1,0} = 0.3$, $p_{0,-1} = 0.3$, $p_{1,-1} = 0.1$ and $h_1=0.1$, $h_{-1} = 0.02$, $h_0 = 0.68$, $v_1 = 0.1$, $v_{-1}=0.03$, $v_0 = 0.67$. The other transition probabilities are zero. 
\end{example}
In Figure~\ref{fig:Ex1aD}, all non-zero transition probabilities, except those for the transitions from a state to itself, are illustrated.

\begin{figure}
\begin{center}
\subfigure[]{\includegraphics[scale=1]{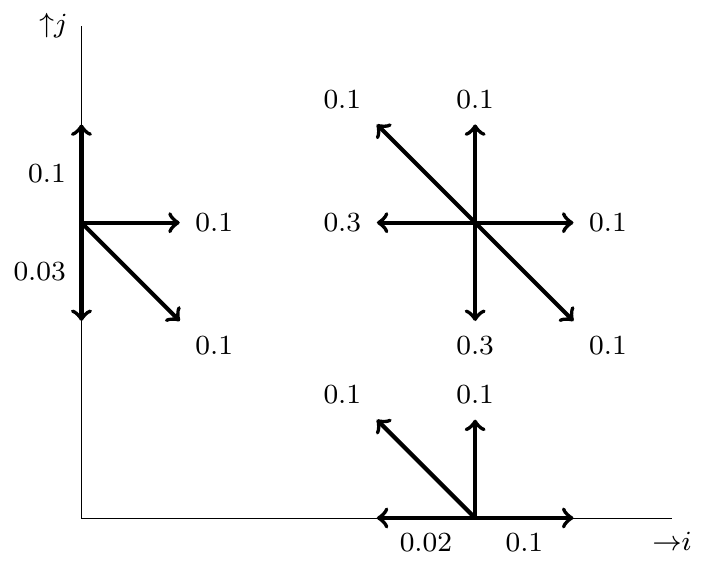}{\label{fig:Ex1aD}}}
 \\
\subfigure[]{\includegraphics[scale=1]{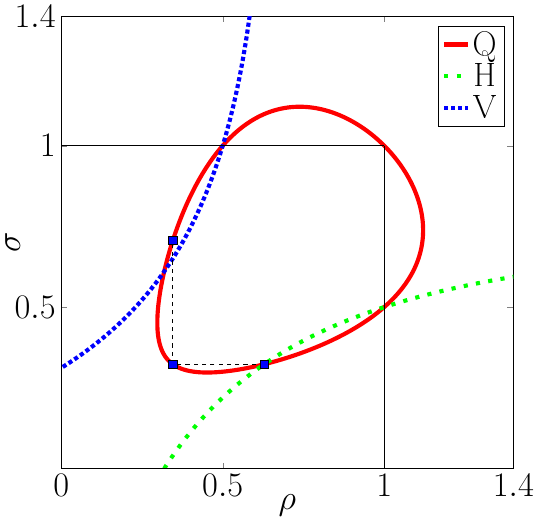}{\label{fig:Ex1bD}}}
\end{center}
\caption{Example~\ref{ex:twoD}.~\subref{fig:Ex1aD} Transition diagram.~\subref{fig:Ex1bD} Algebraic curves $Q$, $H$ and $V$. The geometric terms are denoted by the squares.\label{fig:Ex1D}}
\end{figure}


\begin{figure}
\begin{center}
\subfigure[]{\includegraphics[scale=1]{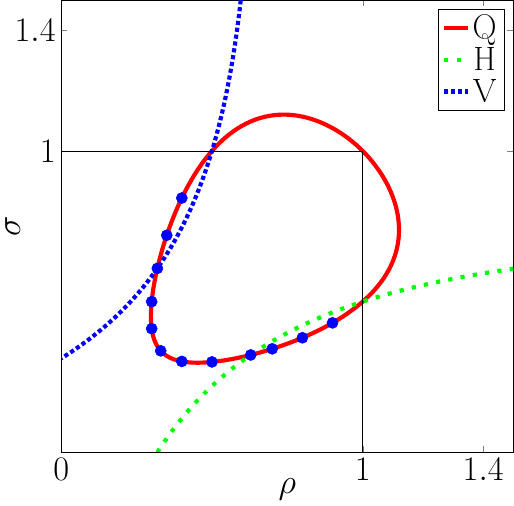}{\label{fig:ebEx1aD}}}
 \\
\subfigure[]{\includegraphics[scale=1]{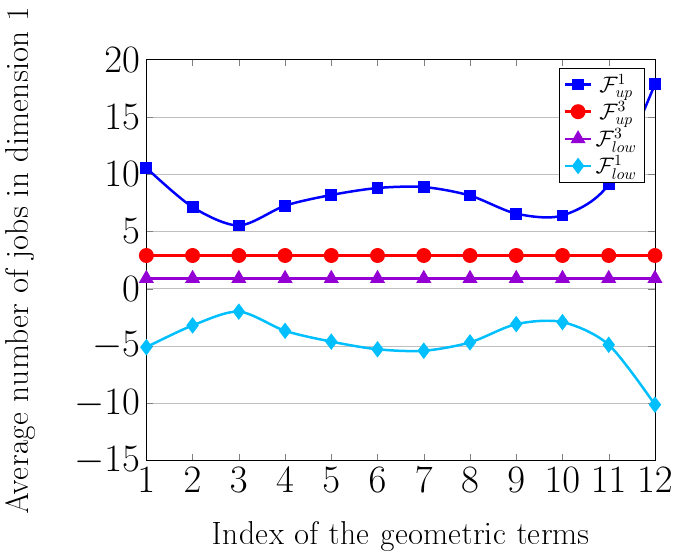}{\label{fig:ebEx1bD}}}
\end{center}
\caption{~\subref{fig:ebEx1aD} The geometric measures from $Q$.~\subref{fig:ebEx1bD} Error bounds for $\mathcal{F}_1$. The $x$-axis in~\subref{fig:ebEx1bD} are the index of $12$ geometric measures in~\subref{fig:ebEx1aD} sorted from left up corner to the right down corner.{\label{fig:ebEx1D}}}
\end{figure}

Using the Detection Algorithm, we conclude that the invariant measure cannot be a sum of geometric terms. Instead, we will find error bounds for $\mathcal{F}_1$. We first obtain error bounds for $\mathcal{F}_1$ using a perturbed random walk of which the invariant measure is of product-form. Figure~\ref{fig:ebEx1aD} shows $12$ different geometric terms which are the invariant measures of the perturbed random walks used to bound $\mathcal{F}_1$. Moreover, we bound $\mathcal{F}_1$ based on a perturbed random walk of which the invariant measure is the sum of the $3$ geometric terms that are depicted as solid squares in Figure~\ref{fig:Ex1bD}. Finally, we find error bounds for $\mathcal{F}_1$ in Figure~\ref{fig:ebEx1bD}. Figure~\ref{fig:ebEx1bD} shows that the minimum of $[\mathcal{F}^1_1]_{up}$ and the maximum of $[\mathcal{F}^1_1]_{low}$, when perturbed random walks of which the invariant measures are of product-form are used, are
\begin{equation*}
\min([\mathcal{F}^1_1]_{up}) = 5.4200, \quad \max([\mathcal{F}^1_1]_{low}) = -1.9527.
\end{equation*}
Note that lower bounds which provide negative values are not directly useful, since $\mathcal{F}$ can always be lower bounded by $0$. However, these bounds indicate the range of errors that our approximation scheme may lead to. The upper and lower bounds for $\mathcal{F}_1$, when the perturbed random walk of which the invariant measure is a sum of $3$ geometric terms which are depicted in Figure~\ref{fig:Ex1bD} is used, are 
\begin{equation*}
[\mathcal{F}^3_1]_{up} = 2.9026, \quad [\mathcal{F}^3_1]_{low} = 0.8964.
\end{equation*}

Clearly, $[\mathcal{F}^3_1]_{up}$ and $[\mathcal{F}^3_1]_{low}$ outperform $[\mathcal{F}^1_1]_{up}$ and $[\mathcal{F}^1_1]_{low}$. From the results above, we conclude that using a perturbed random walk of which the invariant measure is a sum of multiple geometric geometric terms improves the error bounds compared with only using the perturbed random walk of which the invariant measure is of product-form.

In the next example, we consider a discrete-time Markov chain obtained by uniformizing a tandem queue model. The special feature of this model is that, when the first server is idle, the service rate in the second server will be slower.

\begin{example}[Tandem queue with server slow-down]{\label{ex:threeD}}
We have $p_{1,0} = 0.1$, $p_{-1,1} = 0.2$, $p_{0,-1} = 0.3$, $p_{0,0} = 0.4$, $h_1 = 0.1$, $h_{0} = 0.7$, $v_{-1} = 0.03$ and $v_{0} = 0.87$. The other transition probabilities are zero.
\end{example}
In Figure~\ref{fig:tdAD}, all non-zero transition probabilities, except those for the transitions from a state to itself, are illustrated.

\begin{figure}
\begin{center}
\subfigure[]{\includegraphics[scale=1]{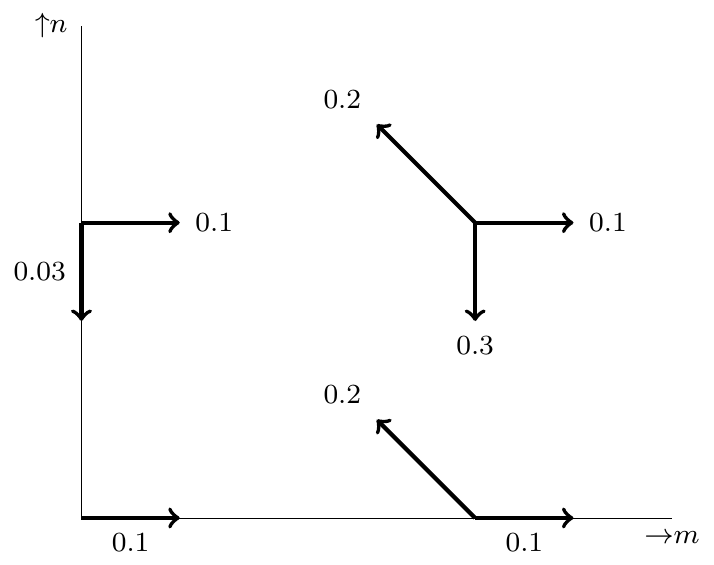}{\label{fig:tdAD}}}
 \\
\subfigure[]{\includegraphics[scale=1]{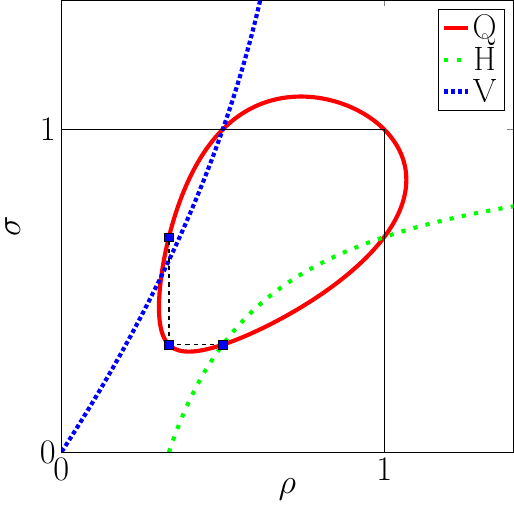}{\label{fig:tdBD}}}
\end{center}
\caption{Example~\ref{ex:threeD}.~\subref{fig:tdAD} Transition diagram.~\subref{fig:tdBD} Algebraic curves $Q$, $H$ and $V$. The geometric terms are denoted by the squares.\label{fig:tdD}}
\end{figure}

Using the Detection Algorithm, we conclude that the invariant measure cannot be a sum of geometric terms. Instead, we find error bounds for $\mathcal{F}_2$. We first obtain error bounds for $\mathcal{F}_2$ using a perturbed random walk of which the invariant measure is of product-form. Figure~\ref{fig:ebEx3aD} shows $12$ different geometric terms which are the invariant measures of the perturbed random walks used to bound $\mathcal{F}_2$. We also bound $\mathcal{F}_2$ based on a perturbed random walk of which the invariant measure is the sum of $3$ geometric terms which are depicted as solid squares in Figure~\ref{fig:tdBD}. Finally, Figure~\ref{fig:ebEx3D} shows bounds for $\mathcal{F}_2$.


\begin{figure}
\begin{center}
\subfigure[]{\includegraphics[scale=1]{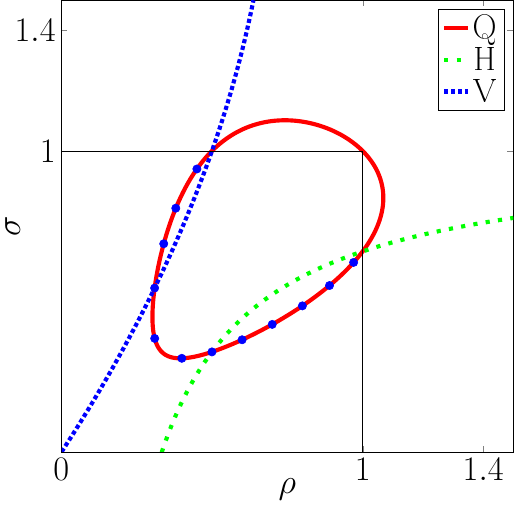}{\label{fig:ebEx3aD}}}
 \\
\subfigure[]{\includegraphics[scale=1]{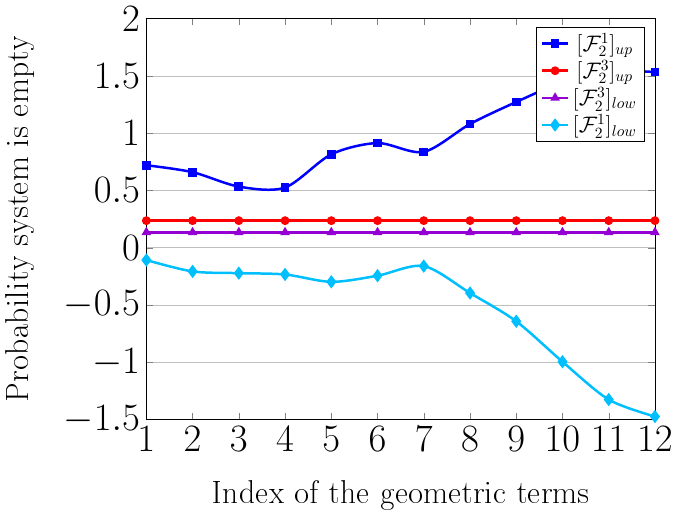}{\label{fig:ebEx3bD}}}
\end{center}
\caption{~\subref{fig:ebEx3aD} The geometric measures from $Q$.~\subref{fig:ebEx3bD} Error bounds for $\mathcal{F}_2$. The $x$-axis in~\subref{fig:ebEx3bD} are the $12$ geometric terms in~\subref{fig:ebEx3aD} sorted from left up corner to the right down corner.\label{fig:ebEx3D}}
\end{figure}

We obtain the error bounds for $\mathcal{F}_2$ in Figure~\ref{fig:ebEx3D}. From Figure~\ref{fig:ebEx3bD}, we have
\begin{equation*}
\min([\mathcal{F}^1_2]_{up}) = 0.5258, \quad \max([\mathcal{F}^1_2]_{low}) = -0.1070.
\end{equation*}
Using the perturbed random walk of which the invariant measure is the sum of $3$ geometric terms depicted as solid squares (see Figure~\ref{fig:tdBD}), we obtain the upper and lower bounds for $\mathcal{F}_2$: 
\begin{equation*}
[\mathcal{F}^3_2]_{up} = 0.2367, \quad [\mathcal{F}^3_2]_{low} = 0.1346,
\end{equation*}
respectively.

Clearly, $[\mathcal{F}^3_2]_{up}$ and $[\mathcal{F}^3_2]_{low}$ outperform $[\mathcal{F}^1_2]_{up}$ and $[\mathcal{F}^1_2]_{low}$.

Note that there is no monotonicity between the number of geometric terms used in the invariant measure of the perturbed random walk and the error bounds of the approximated performance measure. For example, if we use the perturbed random walk of which the invariant measure is induced by $5$ geometric terms, the error bounds for the approximated performance measure is not necessarily better than that obtained via the perturbed random walk of which the invariant measure is induced by $3$ geometric terms. This is the reason why we do not give many numerical illustrations when the invariant measure of the perturbed random walk is induced by $\Gamma$ where $|\Gamma| = 5,7,9, \dots$. 

Finally, note that in out numerical illustrations, perturbations to invariant measures with $3$ geometric terms are better than perturbations to product-form invariant measures. However, there are also examples when perturbations to product-form invariant measures provide the best results.

\section{Conclusion and discussion}{\label{sec:discussionD}}

In this paper, we developed an algorithm to check whether the invariant measure of a given random walk in the quarter-plane is a sum of geometric terms. We also showed how to find the invariant measure explicitly, if the answer is positive. Random walks with such performance measures can be readily evaluated. For the case that the invariant measure of a given random walk is not a sum of geometric terms, we developed an approximation scheme to determine the performance measures. These bounds are determined using a perturbed random walk which differs from the original random walk only in the transitions along the boundaries. We showed numerically that considering a perturbed random walk of which the invariant measure is a sum of geometric terms instead of a perturbed random walk of which the invariant measure is of product-form results in tighter bounds for the performance measures.

In this paper, we assume $p_{1,0} + p_{1,1} + p_{0,1} \neq 0$ because when $p_{1,0} + p_{1,1} + p_{0,1} = 0$, the algebraic curve $Q$ of the random walk has an accumulation point at the origin. In this case, the Detection Algorithm may not stop in finite time. However, this does not prevent us from using the approximation scheme developed in Section~\ref{subsec:ASD} to obtain bounds on the performance measures, \ie using a perturbed random walk of which the invariant measure is a sum of finitely many geometric terms. Therefore, we conclude that our approximation scheme accepts any random walk as an input.

\bibliographystyle{plain}
\bibliography{bibfile} 

\newpage

\appendix 

\section{Proof of Theorem~\ref{thm:oddD}}{\label{app:oddD}}

In order to prove Theorem~\ref{thm:oddD}, we first present a lemma. The vertical branch points which will be used here are defined similarly to the horizontal branch points before. Since the algebraic curve $Q$ has a unique connected component in $[0,\infty)^2$ (see Lemma~7,~\cite{chen2013necessary}), it has two vertical branch points in $[0,\infty)^2$, denoted by $(x_l, y_l)$, $(x_r, y_r)$ with $x_l \geq x_r$.

\begin{lemma}{\label{lem:cycleD}}
Consider the measure $m$ induced by set $\Gamma$, which is the invariant measure of random walk $R$. If we connect every two geometric terms with the same horizontal or vertical coordinates from set $\Gamma$ with a line segment, then these line segments cannot form a cycle.
\end{lemma}

In order to prove Lemma~\ref{lem:cycleD}, we define two types of partition of $Q$, see Figure~\ref{fig:partitionQD}.

\begin{definition}[Partition \rom{1} of $Q$]
The partition $\{Q_{00}, Q_{01}, Q_{10}, Q_{11}\}$ of $Q$ is defined as follows: $Q_{00}$ is the part of $Q$ connecting $(x_l, y_l)$ and $(x_b, y_b)$; $Q_{10}$ is the part of $Q$ connecting $(x_b, y_b)$ and $(x_r, y_r)$; $Q_{01}$ is the part of $Q$ connecting $(x_l, y_l)$ and $(x_t, y_t)$; $Q_{11}$ is the part of $Q$ connecting $(x_r, y_r)$ and $(x_t, y_t)$.
\end{definition}

\begin{definition}[Partition \rom{2} of $Q$]
Let $\{Q_l, Q_c, Q_r\}$ denote a partition of $Q$, where 
\begin{equation*}
\begin{aligned}
Q_l &= \left\{ (x,y)\in Q\; \middle|\; x \leq x_b \right\}, \\
Q_c &= \left\{ (x,y)\in Q\; \middle|\; x_b <x \leq x_t \right\}, \\
Q_r &= \left\{ (x,y)\in Q\; \middle|\; x > x_t \right\}.
\end{aligned}
\end{equation*}
\end{definition}

\begin{figure}
\subfigure[]{\includegraphics{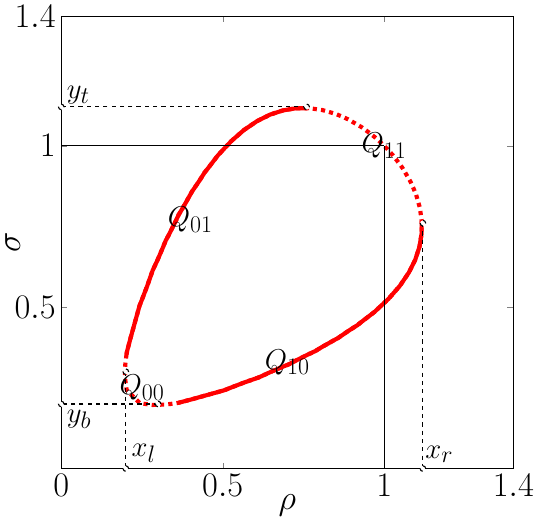}{\label{fig:partition1D}}}
\subfigure[]{\includegraphics{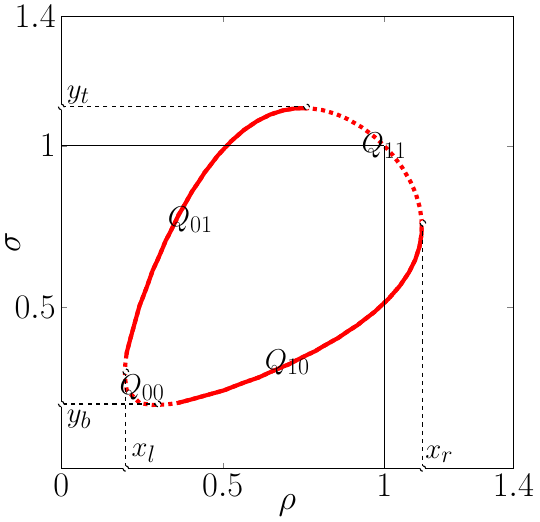}{\label{fig:partition2D}}}
\caption{Different partition of $Q$.~\subref{fig:partition1D} Partition I of $Q$.~\subref{fig:partition2D} Partition II of $Q$.\label{fig:partitionQD}}
\end{figure}

\begin{proof}[Proof of Lemma~\ref{lem:cycleD}]
Denote the two pieces of $Q_c$ in Figure~\ref{fig:partition2D} by $Q_c^t$ and $Q_c^b$ satisfying $\tilde{y} > y$ if $(x,\tilde{y}) \in Q_c^t$ and $(x,y) \in Q_c^b$. Since the algebraic curve $Q$ contains no singularity, because of Theorem~12 from~\cite{chen2013necessary}, $Q_l, Q_c$ and $Q_r$ are all non-empty. 

 In addition, we let $\{\Gamma_1,\dots,\Gamma_K\}$ denote a partition of $\Gamma$, where the elements of $\Gamma_i$ are denoted by $\Gamma_i=\{(\rho_{i,1},\sigma_{i,1}),\dots,(\rho_{i,L(i)},\sigma_{i,L(i)})\}$ and each $\Gamma_i$ satisfies
 \begin{equation} \label{eq:thmpropertyCstairD}
 \begin{IEEEeqnarraybox}[][c]{rClrCl}
 \rho_{i,1} &>& \rho_{i,2},\quad & \sigma_{i,1} &=& \sigma_{i,2}, \\
 \rho_{i,2} &=& \rho_{i,3},\quad & \sigma_{i,2} &>& \sigma_{i,3}, \\ 
 \rho_{i,3} &>& \rho_{i,4},\quad & \sigma_{i,3} &=& \sigma_{i,4}, \\
 &\vdots& & &\vdots& \\
 \rho_{i,L(i)-1} &>& \rho_{i,L(i)},\quad & \sigma_{i,L(i)-1} &=& \sigma_{i,L(i)}.
 \end{IEEEeqnarraybox}
 \end{equation}
In addition the partition $\{\Gamma_1,\dots,\Gamma_K\}$ is maximal in the sense that no $\Gamma_i\cup\Gamma_j$, $i\neq j$ satisfies~\eqref{eq:thmpropertyCstairD}. 

Assume that the line segments, which connect every two geometric terms with the same horizontal or vertical coordinates from set $\Gamma$, form a cycle. Without loss of generality, we will have $\Gamma_1$, $\Gamma_2$ where $|\Gamma_1| > 1$ and $|\Gamma_2| > 1$ such that $\rho_{1,1} = \rho_{2,1}$ and $\rho_{1,1},\rho_{2,1} \in Q_r$. Moreover, either $\rho_{1,1}$ or $\rho_{2,1}$ must be on $Q_{11}$. However, $y_t \geq 1$ and $x_r \geq 1$, by~\cite[Lemma 2.3.8]{fayolle1979two}. Also, using the fact that $Q_{11}$ is monotonic, by Lemma~9 from~\cite{chen2013necessary}, we conclude that $Q_{11}$ is outside of $U$, which contradicts that $m$ is a finite measure.
\end{proof}

We are now ready to prove Theorem~\ref{thm:oddD}.

\begin{proof}[Proof of Theorem~\ref{thm:oddD}]
First, it follows from Theorem~4,~\cite{chen2012invariant} that $\Gamma$ must be a pairwise-coupled set.

From Lemma~\ref{lem:cycleD}, the pairwise-coupled set $\Gamma$ cannot form a cycle. Hence, there must be two geometric terms which do not share the horizontal or vertical coordinate with other geometric terms from set $\Gamma$. We denote these two geometric terms by $(\rho_1, \sigma_1), (\rho_2, \sigma_2)$. 

It follows from Lemma~\ref{lem:6D} that the measure induced by any two geometric terms from set $\Gamma$, which have the same horizontal coordinates, must satisfy the horizontal balance equation. 

Without loss of generality, we assume that $(\rho_1, \sigma_1), (\rho_2, \sigma_2) \in H_{set}$. Thus, for $k = 1,2$, we have
\begin{align}{\label{eq:HandQD}}
B^h(\rho_k, \sigma_k) = \sum_{s=-1}^1 \big(\rho_k^{1-s} h_s+\rho_k^{1-s}\sigma_k p_{s,-1}\big) - \rho_k = 0.  
\end{align}
Hence, for $k = 1,2$, there exists no $(\rho, \sigma) \in \Gamma \backslash (\rho_k, \sigma_k)$ such that $\rho = \rho_k$. Otherwise, the balance for $(\rho_k, \sigma_k)$ and $(\rho, \sigma)$ cannot be satisfied. Moreover, because $\Gamma$ is a pairwise-couple set, there exist a $(\rho, \sigma) \in \Gamma \backslash (\rho_k, \sigma_k)$ such that $\sigma = \sigma_k$. Similar results hold when $(\rho_1, \sigma_1), (\rho_2, \sigma_2) \in V_{set}$ or when $(\rho_1, \sigma_1) \in H_{set}$ and $(\rho_2, \sigma_2) \in V_{set}$.

It can be readily verified that if $(\rho_1, \sigma_1) \in H_{set}$ and $(\rho_2, \sigma_2) \in V_{set}$, then $|\Gamma| = 2k+1$, where $k = 1,2,3, \cdots$. Otherwise, we have $|\Gamma| = 2k$ where $k = 1,2,3, \cdots$.

Finally, if such pairs $(\rho_1, \sigma_1), (\rho_2, \sigma_2)$ are not unique, then, by carefully choosing the coefficients, we find $2$ signed measures to make all balance equations satisfied. However, this contradicts the uniqueness of the invariant measure, which completes the proof.
\end{proof}

\section{Proof of Theorem~\ref{thm:maximalstepsD}}{\label{app:proofD}}

\begin{proof}[Proof of Theorem~\ref{thm:maximalstepsD}]
Similar to the proof of Lemma~\ref{lem:cycleD}, we find $\{\Gamma_1,\dots,\Gamma_K\}$ which are defined in~\eqref{eq:thmpropertyCstairD}.

First, we prove $L(i)<\infty$ by demonstrating that 
\begin{equation*}
 |\Gamma_i\cap Q_l| \leq 1,\quad \left|\Gamma_i\cap Q_c\right| <\infty,
\quad \left|\Gamma_i\cap Q_r\right| \leq 1.
\end{equation*}
Suppose that $\left|\Gamma_i\cap Q_r\right| \geq 2$. Then there exist $(\rho, \sigma)$ and $(\tilde{\rho}, \tilde{\sigma})$ on $Q_{11}$ or $Q_{10}$ satisfying $\tilde{\sigma} = \sigma$. This contradicts Lemma~9,~\cite{chen2013necessary}, which indicates the monotonicity of $Q_{11}$ and $Q_{10}$.
 
Therefore, $\left|\Gamma_i\cap Q_r\right| \leq 1$. Similarly, we show $\left|\Gamma_i\cap Q_l\right| \leq 1$.  

Next, we prove that $\sigma_{i,j+2}\leq \sigma_{i,j}-\min(D_1, D_2)$ where
\begin{equation*}
D_1 = \frac{\Delta_y(x_b)}{\sum_{s = -1} ^{1} p_{s,-1} x_t^{1 - s}}, D_2 = \frac{\Delta_y(x_t)}{\sum_{s = -1} ^{1} p_{s,-1} x_t^{1 - s}},
\end{equation*}
for three consecutive elements in $\left|\Gamma_i\cap Q_c\right|$, $(\rho_{i,j}, \sigma_{i,j})$, $(\rho_{i+1, j+1}, \sigma_{i+1, j+1})$ and $(\rho_{i+2, j+2}, \sigma_{i+2, j+2})$ satisfying
 \begin{equation*}
 \begin{IEEEeqnarraybox}[][c]{rClrCl}
 \rho_{i,j} &>& \rho_{i,j+1},\quad & \sigma_{i,j} &=& \sigma_{i,j+1}, \\
 \rho_{i,j+1} &=& \rho_{i,j+2},\quad & \sigma_{i,j+1} &>& \sigma_{i,j+2}.
 \end{IEEEeqnarraybox}
 \end{equation*} 
Note that $\Delta_y(x) > 0$ and $\Delta_y(x)$ has at most one stationary point where the derivative is $0$ for $x \in (x_b, x_t)$ because $\Delta_y(x)$ is continuous over $x$ and $\Delta_y(x) = 0$ has $4$ real solutions due to Lemma~4,~\cite{chen2013necessary}. We obtain that $\Delta_y(x) \geq \min(\Delta_y(x_b), \Delta_y(x_t))$. Moreover, it can be readily verified that $\sum_{s = -1}^{1} p_{s,-1} x^{1 - s}$ is monotonically increasing in $x$ for $x \in (x_b, x_t)$. Therefore, we have
\begin{equation}{\label{eq:lengthrelationD}}
\frac{\Delta_y(x)}{\sum_{s = -1}^{1} p_{s,-1} x^{1 - s}} \geq \min\left(\frac{\Delta_y(x_b)}{\sum_{s = -1}^{1} p_{s,-1} x_t^{1 - s}}, \frac{\Delta_y(x_t)}{\sum_{s = -1}^{1} p_{s,-1} x_t^{1 - s}}\right),
\end{equation}
for $x\in (x_b, x_t)$.

Notice that the left side of equation~\eqref{eq:lengthrelationD} is the distance between two intersections of $Q$ and a vertical line, \ie $\frac{\Delta_y(a)}{\sum_{s = -1}^{1} p_{s,-1} a^{1 - s}}$ is the distance between two intersections of $Q$ and line $x = a$. Therefore, we conclude that $\sigma_{i,j+2}\leq \sigma_{i,j}-\min(D_1, D_2)$ where
\begin{equation*}
D_1 = \frac{\Delta_y(x_b)}{\sum_{s = -1} ^{1} p_{s,-1} x_t^{1 - s}}, D_2 = \frac{\Delta_y(x_t)}{\sum_{s = -1} ^{1} p_{s,-1} x_t^{1 - s}}.
\end{equation*}

Next, we show that if $K > 2$, then there exists a $(\rho, \sigma) \in \Gamma$ such that $\rho > 1$ or $\sigma > 1$. Without loss of generality, we assume $K=3$. Observe that  $\{\Gamma_1,\Gamma_2,\Gamma_3\}$ forms a pairwise-coupled set. Using from the above that $|\Gamma_i|<\infty$ for $i=1,2,3$, we must have $\rho_{1,L(1)} = \rho_{2,L(2)}$ with $\rho_{1, L(1)}, \rho_{2, L(2)} \in Q_l$ and $\rho_{2, 1} = \rho_{3,1}$ with $\rho_{2, 1},\rho_{3,1} \in Q_r$ after a proper ordering of $\{\Gamma_1,\Gamma_2,\Gamma_3\}$. Moreover, either $\rho_{2,1}$ or $\rho_{3,1}$ must be on $Q_{11}$. However, $y_t \geq 1$ and $x_r \geq 1$ due to~\cite[Lemma 2.3.8]{fayolle1979two}. Also, using the fact that $Q_{11}$ is monotonic, due to Lemma~9 from~\cite{chen2013necessary}, we conclude that $Q_{11}$ is outside of $U$. Hence, there exists a $(\rho, \sigma) \in \Gamma$ such that $\rho > 1$ or $\sigma > 1$.

When $K \leq 2$, we know that the distance between two intersections of $Q$ and a vertical line $x = a$ where $a \in (x_b, x_t)$ is at least $\min(D_1, D_2)$. 

Therefore, we conclude that if  
\begin{equation*}
|\Gamma| > M(R) =  \frac{6}{\min(D_1, D_2)} + 4,
\end{equation*}
then there exists a $(\rho, \sigma) \in \Gamma$ such that $\rho > 1$ or $\sigma > 1$.
We know from Theorem~12,~\cite{chen2013necessary}, that the algebraic curve $Q$ can only have an accumulation point at the origin when $p_{1,0} + p_{1,1} + p_{0,1} = 0$. Hence, when $p_{1,0} + p_{1,1} + p_{0,1} \neq 0$, we have $x_b > 0$ and $x_t > 0$. This means that $D_1 > 0$ and $D_2> 0$. Therefore, we have $M(R) < \infty$.
\end{proof}

\end{document}